\documentclass[11pt]{amsart}
\usepackage[top=1.5in, bottom=1.5in, left=1.4in, right=1.4in]{geometry} 
\usepackage{color}
\usepackage{graphicx}
\usepackage{amssymb}
\usepackage{epstopdf}
\usepackage{amsthm}
\usepackage{hyperref}
\usepackage[table]{xcolor}
\usepackage{array}
\usepackage{mathtools}
\usepackage{enumerate}
\usepackage{lscape}
\usepackage{verbatim}
\usepackage{color}
\usepackage{multicol}

\allowdisplaybreaks

\usepackage[all]{xy}
\usepackage{graphicx}  
\usepackage{tikz}
\usetikzlibrary{decorations.pathreplacing}
\usepackage{tkz-graph}
\usepackage[all]{xy}
\usepackage{graphicx}  
\usetikzlibrary{arrows}
\usetikzlibrary{decorations.markings}

\pgfarrowsdeclare{varstealth}{varstealth}
 {
  \pgfarrowsleftextend{-5.59\pgflinewidth}
  \pgfarrowsrightextend{4.5\pgflinewidth}}
 {
  \pgfpathmoveto{\pgfpoint{4.5\pgflinewidth}{0pt}} %Point5
  \pgfpathlineto{\pgfpoint{-0.63\pgflinewidth}{1.09\pgflinewidth}} %Point4
  \pgfpathlineto{\pgfpoint{-5.48\pgflinewidth}{3.06\pgflinewidth}} %Point3
  \pgfpathlineto{\pgfpoint{-5.59\pgflinewidth}{3.0\pgflinewidth}}%Point2
  \pgfpathlineto{\pgfpoint{-3.95\pgflinewidth}{0pt}} %Point1
  \pgfpathlineto{\pgfpoint{-5.59\pgflinewidth}{-3.0\pgflinewidth}}%Point8
  \pgfpathlineto{\pgfpoint{-5.48\pgflinewidth}{-3.06\pgflinewidth}}%Point7
  \pgfpathlineto{\pgfpoint{-0.63\pgflinewidth}{-1.09\pgflinewidth}}%Point6
  \pgfusepathqfill
 }

\def\SetFancyGraph {
	\SetVertexMath
	\GraphInit[vstyle=Art]
	\SetUpVertex[MinSize=2pt]
	\SetVertexLabel
	\tikzset{VertexStyle/.style = {shape = circle,shading = ball,ball color = black,inner sep = 1.5pt}}
	\SetUpEdge[color=black]
	\tikzset{>=varstealth}
	\tikzset{->-/.style={decoration={ markings, mark=at position 0.8 with {\arrow{>}}},postaction={decorate}}}
	\tikzset{->--/.style={decoration={ markings, mark=at position 0.55 with {\arrow{>}}},postaction={decorate}}}
}

\newtheorem{thm}{Theorem}
\newtheorem{lemma}[thm]{Lemma}

\newtheorem{cor}[thm]{Corollary}
\newtheorem{prop}[thm]{Proposition}
\newtheorem{conj}[thm]{Conjecture}

\newtheorem{Definition}[thm]{Definition}

\newtheorem{Example}[thm]{Example}

\newtheorem{Remark}[thm]{Remark}
\newenvironment{remark}
  {\begin{Remark}\rm}{\end{Remark}}
  
\numberwithin{equation}{section}

\usepackage{etoolbox}
\apptocmd{\sloppy}{\hbadness 10000\relax}{}{}

\title{Sorting via chip-firing}
\author{Sam Hopkins}
\author{Thomas McConville}
\author{James Propp}

\begin{document}

\begin{abstract}
We investigate a variant of the chip-firing process on the infinite path graph~$\mathbb{Z}$: rather than treating the chips as indistinguishable, we label them with positive integers.  To fire an unstable vertex, i.e.\ a vertex with more than one chip, we choose any two chips at that vertex and move the lesser-labeled chip to the left and the greater-labeled chip to the right.  This labeled version of the chip-firing process exhibits a remarkable confluence property, similar to but subtler than the confluence that prevails for unlabeled chip-firing: when all chips start at the origin and the number of chips is even, the chips always end up in sorted order.  Our proof of sorting relies upon an independently interesting lemma concerning unlabeled chip-firing which says that stabilization preserves a natural partial order on configurations. We also discuss some extensions of this sorting phenomenon to other graphs (variants of the infinite path), to other initial configurations, and to other Cartan-Killing types.
\end{abstract}

\maketitle

\section{Introduction} \label{sec:intro}

We introduce a labeled version of the chip-firing process. The chip-firing process is a discrete dynamical system that takes place on a graph. The name ``chip-firing'' was coined by Bj\"{o}rner, Lovasz, and Shor~\cite{bjorner1991chip}, but in fact this process is essentially the same as the abelian sandpile model introduced by Bak, Tang, and Wiesenfeld~\cite{bak1987self} and further developed by Dhar~\cite{dhar1990self}~\cite{dhar1999abelian}. (See also the work of Engel~\cite{engel1976why}.) Here ``abelian'' means that the order in which certain local moves are carried out has no effect on the final state. This property is also often called ``confluence'' in the context of abstract rewriting systems~\cite{huet1980confluent}.\footnote{We will often use ``confluent'' in a slightly non-standard sense, where it might be more correct to say ``confluent and terminating.'' See the proof of Lemma~\ref{lem:stab} for technical definitions of these properties.} Dhar views the abelian sandpile model as a prototype for networks of communicating processors that achieve predictable results despite a lack of global synchronization; this point of view has been developed by Levine and his coauthors~\cite{bond2016abelian1}~\cite{bond2016abelian2}~\cite{bond2016abelian3}~\cite{holroyd2015abelian}, who introduced a broad family of such networks, called ``abelian networks,'' to which the abelian sandpile model belongs. For more background on sandpiles, consult the short survey article~\cite{levine2010sandpile} or the upcoming book~\cite{corry2017divisors}.

Bj\"{o}rner, Lovasz, and Shor were directly inspired by papers of Spencer~\cite{spencer1986balancing} and Anderson et al.~\cite{anderson1989disks} that investigated chip-firing in the special case of the infinite path graph~$\mathbb{Z}$ (which has vertex set $\mathbb{Z}$ with $i$ and $j$ joined by an edge whenever $|i-j| = 1$). In fact, it is in exactly this special case of the infinite path graph~$\mathbb{Z}$ that we introduce labeled chip-firing. Here is how the labeled chip-firing process on $\mathbb{Z}$ works: we start with~$n$ labeled chips~$(1),(2),\ldots,(n)$ at the origin; at each step we choose any two chips~$(a)$ and~$(b)$ with $a < b$ that occupy the same vertex~$i$ and \emph{fire} these chips together, moving~$(a)$ to vertex~$i-1$ and $(b)$ to vertex~$i+1$; we keep carrying out firings until no chips can fire. For example, suppose $n=4$, so that we start from the following configuration (where we draw~$\mathbb{Z}$ in the plane as a number line in the usual way, with~$i-1$ to the left of $i$):
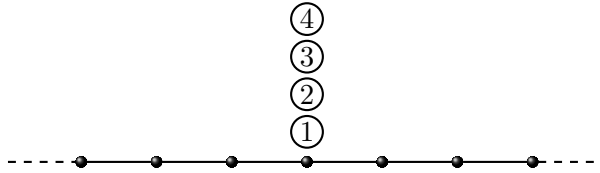
\begin{center}
\begin{tikzpicture}[scale=1]
	\SetFancyGraph
	\Vertex[NoLabel,x=-4,y=0,empty]{-4}
	\Vertex[NoLabel,x=-3,y=0]{-3}
	\Vertex[NoLabel,x=-2,y=0]{-2}
	\Vertex[NoLabel,x=-1,y=0]{-1}
	\Vertex[NoLabel,x=0,y=0]{0}
	\Vertex[NoLabel,x=1,y=0]{1}
	\Vertex[NoLabel,x=2,y=0]{2}
	\Vertex[NoLabel,x=3,y=0]{3}
	\Vertex[NoLabel,x=4,y=0,empty]{4}
	\Edges[style={thick}](-3,-2,-1,0,1,2,3)
	\Edges[style={thick,dashed}](-4,-3)
	\Edges[style={thick,dashed}](3,4)
	\node [circle,draw,thick,inner sep=0,minimum size=12pt] at (0,0.4) {1};
	\node [circle,draw,thick,inner sep=0,minimum size=12pt] at (0,0.9) {2};
	\node [circle,draw,thick,inner sep=0,minimum size=12pt] at (0,1.4) {3};
	\node [circle,draw,thick,inner sep=0,minimum size=12pt] at (0,1.9) {4};
\end{tikzpicture}
\end{center}
By firing chips~$(1)$ and~$(2)$ we reach the following configuration:
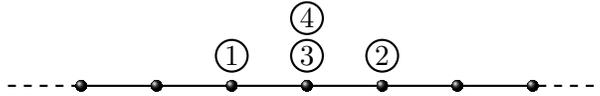
\begin{center}
\begin{tikzpicture}[scale=1]
	\SetFancyGraph
	\Vertex[NoLabel,x=-4,y=0,empty]{-4}
	\Vertex[NoLabel,x=-3,y=0]{-3}
	\Vertex[NoLabel,x=-2,y=0]{-2}
	\Vertex[NoLabel,x=-1,y=0]{-1}
	\Vertex[NoLabel,x=0,y=0]{0}
	\Vertex[NoLabel,x=1,y=0]{1}
	\Vertex[NoLabel,x=2,y=0]{2}
	\Vertex[NoLabel,x=3,y=0]{3}
	\Vertex[NoLabel,x=4,y=0,empty]{4}
	\Edges[style={thick}](-3,-2,-1,0,1,2,3)
	\Edges[style={thick,dashed}](-4,-3)
	\Edges[style={thick,dashed}](3,4)
	\node [circle,draw,thick,inner sep=0,minimum size=12pt] at (-1,0.4) {1};
	\node [circle,draw,thick,inner sep=0,minimum size=12pt] at (1,0.4) {2};
	\node [circle,draw,thick,inner sep=0,minimum size=12pt] at (0,0.4) {3};
	\node [circle,draw,thick,inner sep=0,minimum size=12pt] at (0,0.9) {4};
\end{tikzpicture}
\end{center}
Then by firing~$(3)$ and~$(4)$ together we reach this configuration:
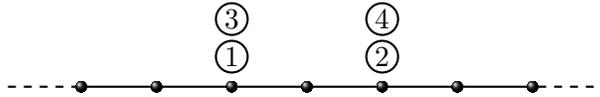
\begin{center}
\begin{tikzpicture}[scale=1]
	\SetFancyGraph
	\Vertex[NoLabel,x=-4,y=0,empty]{-4}
	\Vertex[NoLabel,x=-3,y=0]{-3}
	\Vertex[NoLabel,x=-2,y=0]{-2}
	\Vertex[NoLabel,x=-1,y=0]{-1}
	\Vertex[NoLabel,x=0,y=0]{0}
	\Vertex[NoLabel,x=1,y=0]{1}
	\Vertex[NoLabel,x=2,y=0]{2}
	\Vertex[NoLabel,x=3,y=0]{3}
	\Vertex[NoLabel,x=4,y=0,empty]{4}
	\Edges[style={thick}](-3,-2,-1,0,1,2,3)
	\Edges[style={thick,dashed}](-4,-3)
	\Edges[style={thick,dashed}](3,4)
	\node [circle,draw,thick,inner sep=0,minimum size=12pt] at (-1,0.4) {1};
	\node [circle,draw,thick,inner sep=0,minimum size=12pt] at (1,0.4) {2};
	\node [circle,draw,thick,inner sep=0,minimum size=12pt] at (-1,0.9) {3};
	\node [circle,draw,thick,inner sep=0,minimum size=12pt] at (1,0.9) {4};
\end{tikzpicture}
\end{center}
After firing~$(1)$ and~$(3)$, and then firing~$(2)$ and~$(4)$, in two more steps we reach the following configuration:
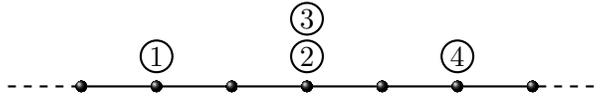
\begin{center}
\begin{tikzpicture}[scale=1]
	\SetFancyGraph
	\Vertex[NoLabel,x=-4,y=0,empty]{-4}
	\Vertex[NoLabel,x=-3,y=0]{-3}
	\Vertex[NoLabel,x=-2,y=0]{-2}
	\Vertex[NoLabel,x=-1,y=0]{-1}
	\Vertex[NoLabel,x=0,y=0]{0}
	\Vertex[NoLabel,x=1,y=0]{1}
	\Vertex[NoLabel,x=2,y=0]{2}
	\Vertex[NoLabel,x=3,y=0]{3}
	\Vertex[NoLabel,x=4,y=0,empty]{4}
	\Edges[style={thick}](-3,-2,-1,0,1,2,3)
	\Edges[style={thick,dashed}](-4,-3)
	\Edges[style={thick,dashed}](3,4)
	\node [circle,draw,thick,inner sep=0,minimum size=12pt] at (-2,0.4) {1};
	\node [circle,draw,thick,inner sep=0,minimum size=12pt] at (0,0.4) {2};
	\node [circle,draw,thick,inner sep=0,minimum size=12pt] at (0,0.9) {3};
	\node [circle,draw,thick,inner sep=0,minimum size=12pt] at (2,0.4) {4};
\end{tikzpicture}
\end{center}
Finally, by firing~$(2)$ and~$(3)$ we reach the following \emph{stable} configuration, which has no more possible firings:
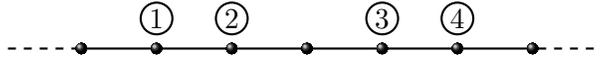
\begin{center}
\begin{tikzpicture}[scale=1]
	\SetFancyGraph
	\Vertex[NoLabel,x=-4,y=0,empty]{-4}
	\Vertex[NoLabel,x=-3,y=0]{-3}
	\Vertex[NoLabel,x=-2,y=0]{-2}
	\Vertex[NoLabel,x=-1,y=0]{-1}
	\Vertex[NoLabel,x=0,y=0]{0}
	\Vertex[NoLabel,x=1,y=0]{1}
	\Vertex[NoLabel,x=2,y=0]{2}
	\Vertex[NoLabel,x=3,y=0]{3}
	\Vertex[NoLabel,x=4,y=0,empty]{4}
	\Edges[style={thick}](-3,-2,-1,0,1,2,3)
	\Edges[style={thick,dashed}](-4,-3)
	\Edges[style={thick,dashed}](3,4)
	\node [circle,draw,thick,inner sep=0,minimum size=12pt] at (-2,0.4) {1};
	\node [circle,draw,thick,inner sep=0,minimum size=12pt] at (-1,0.4) {2};
	\node [circle,draw,thick,inner sep=0,minimum size=12pt] at (1,0.4) {3};
	\node [circle,draw,thick,inner sep=0,minimum size=12pt] at (2,0.4) {4};
\end{tikzpicture}
\end{center}
In this process we made several arbitrary choices of which pairs of chips to fire. As it turns out, no matter what choices we made we would have always reached this same stable configuration where the chips appear in sorted order from left to right. This is a confluence result which says that the divergent paths our process can take must at some later point come back together. It is well known that confluence holds for the unlabeled chip-firing process (on any graph and for any configuration with sufficiently few chips to guarantee that stabilization is possible at all): this is one of the basic results of~\cite{bjorner1991chip}. But confluence for the labeled chip-firing process is much subtler. Indeed, not all initial configurations are confluent in the labeled chip-firing process. This means in particular that the theory of abelian networks does not automatically apply to our situation. To see that not all initial configurations are confluent, consider the configuration that has three chips at the origin rather than four. We can fire~$(1)$ and~$(2)$, or fire~$(1)$ and~$(3)$, or fire~$(2)$ and~$(3)$; in all three cases, the result is a stable configuration, but the labeled chips end up at different vertices, so confluence does not hold in this case. More generally, if we start with an \emph{odd} number of labeled chips at the origin, we can make sure that any preselected chip never moves away from the origin, so confluence does not hold. Our main result, proved in Section~\ref{sec:main}, says that the labeled chip-firing process on $\mathbb{Z}$ is confluent, and in particular sorts the chips, as long as the number of chips is \emph{even}. 

In Section~\ref{sec:extensions} we discuss some extensions of this sorting phenomenon to other graphs (variants of the infinite path) and to other initial configurations. We also show how our main result can be related to Type A root systems and briefly discuss extensions of the problem to other types. Section~\ref{sec:extensions} contains some results, but also many conjectures and threads of future research.

We use the following notation throughout: $\mathbb{R}$ is the set of real numbers, $\mathbb{Z}$ is the set of integers, $\mathbb{N} \coloneqq \{0,1,2,\ldots\}$ is the set of natural numbers, $\mathbb{Z}_{>0} \coloneqq \{1,2,\ldots\}$ is the set of positive integers; for $a,b \in \mathbb{Z}$ we set $[a,b] \coloneqq \{a,a+1,\ldots,b\}$ (which is $\varnothing$ if $a > b$), for $n \in \mathbb{N}$ we set $[n] \coloneqq [1,n]$; and for $x\in\mathbb{R}$ we use $\lfloor x \rfloor$ to denote the greatest integer less than or equal to~$x$.

{\bf Acknowledgements}: We thank Pedro Felzenszwalb, Pavel Galashin, Caroline Klivans, Gregg Musiker, and Peter Winkler for useful comments and discussion. We thank Richard Stanley and the Cambridge Combinatorics and Coffee Club for providing a space to present some initial experimental results related to this research. The first author was supported by NSF grant~\#1122374. The third author was supported by NSF grant~\#1001905.

\section{Main result} \label{sec:main}

One of the main ways we will understand this labeled chip-firing process is by relating it to the usual unlabeled chip-firing process. Therefore, we first review unlabeled chip-firing on the graph $\mathbb{Z}$.  A \emph{configuration of unlabeled chips} on $\mathbb{Z}$ is some assignment of a finite number of indistinguishable chips to the vertices of $\mathbb{Z}$. All configurations, both labeled and unlabeled, will be on $\mathbb{Z}$ in what follows in this section. We use lowercase letters for unlabeled configurations. Formally, we treat an unlabeled configuration $c$ as a function $c \colon \mathbb{Z} \to \mathbb{N}$ with $\sum_i c(i) < \infty$ and we think of $c$ as having $c(i)$ chips at~$i$. We use $\mathrm{supp}(c)$ to denote the \emph{support} of $c$, i.e., $\mathrm{supp}(c) \coloneqq \{i \in \mathbb{Z}\colon c(i) \geq 1\}$. We write $\mathrm{max}(c) \coloneqq \mathrm{max}(\mathrm{supp}(c))$ and $\mathrm{min}(c) \coloneqq \mathrm{min}(\mathrm{supp}(c))$. As is customary, we use the convention~$\mathrm{max}(\varnothing) \coloneqq -\infty$ and~$\mathrm{min}(\varnothing) \coloneqq \infty$.  If $c$ is a configuration and~$i \in \mathbb{Z}$ is some vertex such that~$c$ has at least two chips at $i$, we may perform a \emph{chip-firing move} at $i$, which moves one chip at~$i$ leftward one vertex and one chip at~$i$ rightward one vertex. If the resulting configuration is $c'$ then in this case we also say that $c'$ is obtained from~$c$ by \emph{firing at vertex~$i$}. We write~$c \to d$ to mean that $d$ is obtained from~$c$ by some sequence of (zero or more) chip-firing moves. We say $c$ is \emph{stable} if we cannot perform any chip-firing moves on $c$, that is, if~$c \to d$ implies that~$c=d$. Of course, the map~$c \mapsto \mathrm{supp}(c)$ is a bijection between stable configurations of $n$ chips and subsets of~$\mathbb{Z}$ of size~$n$. 

We now define some useful statistics of configurations:
\begin{align*}
\varphi_{\ell}(c)  &\coloneqq \sum_{i \leq \ell} (i - \ell-1)\cdot c(i) \; \; \textrm{for all $\ell \in \mathbb{Z}$}; \\
\varphi_{\infty}(c) &\coloneqq \sum_{i \in \mathbb{Z}} i \cdot c(i); \\
\varphi^2_{\infty}(c) &\coloneqq \sum_{i \in \mathbb{Z}} i^2 \cdot c(i); \\
\gamma(c) &\coloneqq \#\{i \in \mathbb{Z}\colon \mathrm{min}(c) \leq i \leq \mathrm{max}(c) \textrm{ and } \{i,i+1\} \cap \mathrm{supp}(c) = \varnothing\}.
\end{align*}

\begin{prop} \label{prop:statistics}
Suppose $c'$ is obtained from $c$ by firing at vertex $j \in \mathbb{Z}$; then:
\begin{enumerate}
\item $\varphi_{\ell}(c') = \begin{cases} \varphi_{\ell}(c) - 1 &\textrm{if $j = \ell+1$}; \\ \varphi_{\ell}(c) &\textrm{otherwise}; \end{cases}$
\item $\varphi_{\infty}(c') = \varphi_{\infty}(c)$;
\item $\varphi^2_{\infty}(c') = \varphi^2_{\infty}(c) + 2$;
\item $\gamma(c') \leq \gamma(c)$.
\end{enumerate}
\end{prop}
\begin{proof}
These are routine to verify.
\end{proof}

The following confluence property of the chip-firing process is well-known; for instance, it can be deduced from the main results of~\cite{bjorner1991chip}. It was also proven in a special case by Anderson et al.~\cite{anderson1989disks}, but the proof in general is more or less the same as that special case. We include a short proof, based on Newman's lemma~\cite{newman1942theories}~\cite[Lemma~2.4]{huet1980confluent} (a.k.a.~the diamond lemma), for completeness.

\begin{lemma} \label{lem:stab}
For any configuration $c$, there is a unique stable $d$ with $c \to d$.
\end{lemma}
\begin{proof}
Let us write $c \xrightarrow{F} c'$ if $c'$ is obtained from $c$ by a single chip-firing move. Then~$\to$ is the reflexive transitive closure of $\xrightarrow{F}$. Thus, Newman's lemma~\cite{newman1942theories}~\cite[Lemma 2.4]{huet1980confluent} says that $\xrightarrow{F}$ is \emph{confluent}, meaning that whenever $c \to d$ and $c \to d'$, there exists $d''$ such that $d \to d''$ and $d' \to d''$, as long as the following two conditions hold:
\begin{itemize}
\item $\xrightarrow{F}$ is \emph{locally confluent}, i.e., whenever $c \xrightarrow{F} d$ and $c \xrightarrow{F} d'$, there exists $d''$ such that $d \to d''$ and $d' \to d''$;
\item $\xrightarrow{F}$ is \emph{terminating} (also sometimes called \emph{noetherian}), i.e., there is no infinite sequence $c_0 \xrightarrow{F} c_1 \xrightarrow{F} c_2 \xrightarrow{F} \cdots$. 
\end{itemize}
It easy to see that $\xrightarrow{F}$ is locally confluent. Indeed, suppose $d$ is obtained from $c$ by firing at vertex~$i$ and $d'$ is obtained by firing at vertex~$j$. If $i=j$ then $d=d'$ and so we can take $d'' \coloneqq d$. On the other hand, if $i \neq j$ then $j$ remains fireable in $d$ and we can let $d''$ be the result of firing~$j$ in $d$.

Now we prove that $\xrightarrow{F}$ is terminating. To show $\xrightarrow{F}$ is terminating it suffices to show the following:
\begin{itemize}
\item $\xrightarrow{F}$ is \emph{acyclic}, i.e., there is no cycle $c \xrightarrow{F} c_1 \xrightarrow{F} c_2 \xrightarrow{F}  \cdots \xrightarrow{F} c_k = c$ for any $k\geq 1$;
\item $\xrightarrow{F}$ is is \emph{globally finite}, i.e., for any $c$ there are only finitely many $d$ with $c \to d$.
\end{itemize}
 Let $c$ be a configuration of $n$ chips. If $n=0$ there is nothing to show, so suppose that~$n > 1$.  To see that there cannot be a cycle of firings from $c$ to itself, recall from Proposition~\ref{prop:statistics} that $c \xrightarrow{F} d$ means $\varphi^2_{\infty}(d) = \varphi^2_{\infty}(c) + 2$. So any purported cycle at configuration~$c$ of length $k \geq 1$ would lead to $\varphi^2_{\infty}(c) = \varphi^2_{\infty}(c) + 2k$, a contradiction. Now let us show that there are only finitely many $d$ with $c \to d$. Recall from Proposition~\ref{prop:statistics} that $c \to d$ implies $\varphi_{\infty}(d) = \varphi_{\infty}(c)$. This in particular implies that there is some~$B_1 \in \mathbb{N}$ such that $\mathrm{min}(d) \leq B_1$ and~$\mathrm{max}(d) \geq -B_1$ for any $c \to d$. Also recall from Proposition~\ref{prop:statistics} that $c \to d$ implies $\gamma(d) \leq \gamma(c)$. This in particular implies that there is~$B_2 \in \mathbb{N}$ such that $\mathrm{max}(d) - \mathrm{min}(d) \leq B_2$ for all $c \to d$. Altogether, we can conclude that there is some~$B_3 \in \mathbb{N}$ such that $\mathrm{max}(d) \leq B_3$ and $\mathrm{min}(d) \geq -B_3$ for all~$c \to d$. Clearly there are only finitely many configurations $d$ of $n$ chips satisfying $\mathrm{max}(d) \leq B_3$ and~$\mathrm{min}(d) \geq -B_3$ so indeed there can only be finitely many $d$ with $c \to d$.

That $\xrightarrow{F}$ is terminating directly implies that for each $c$ there exists some stable~$d$ with $c \to d$. Finally, the confluence of $\xrightarrow{F}$ implies that this $d$ must be unique.
\end{proof}

We use $\widetilde{c}$ to denote the \emph{stabilization} of $c$, i.e.~the unique stable~$d$ with $c \to d$ guaranteed by Lemma~\ref{lem:stab}. A fact (which follows immediately from Lemma~\ref{lem:stab}) that we will use over and over again is that if $c \to d$ then $\widetilde{d} = \widetilde{c}$.

Let us introduce some notation for specific configurations of unlabeled chips. For two configurations~$c$ and~$d$, let us use~$c+d$ to denote their \emph{sum}, i.e., the configuration with $(c+d)(i) = c(i) + d(i)$ for all~$i \in \mathbb{Z}$. It is clear that if $c \to c'$ then $c + d \to c' + d$. For~$n \in \mathbb{N}$ and a configuration $c$, we use the shorthand $nc \coloneqq \overbrace{c + c + \cdots c}^{n \textrm{ terms }}$. For $i \in \mathbb{Z}$ we let $\delta_i$ denote the configuration that has a single chip at $i$ and no other chips; in other words, $\delta_i$ is the unique stable configuration with $\mathrm{supp}(\delta_i) = \{i\}$. For $i,j \in \mathbb{Z}$, we let~$\delta_{[i,j]}$ denote the configuration that has one chip at vertex~$k$ for all $i \leq k \leq j$ and no other chips; in other words $\delta_{[i,j]}$ is the unique stable configuration with $\mathrm{supp}(\delta_{[i,j]}) = [i,j]$. Note in particular that $\delta_i = \delta_{[i,i]}$.

We now describe some formulas for specific stabilizations that will be needed later.

\begin{prop} \label{prop:linestab}
Suppose that $c = \delta_{[a+1,b-1]} + \delta_i$, where $a,b,i \in \mathbb{Z}$ satisfy $a < b$, $a \leq i$, and~$i \leq b$. Then we have~$\widetilde{c} = \delta_{[a,a+b-i-1]} + \delta_{[a+b-i+1,b]}$.
\end{prop}
\begin{proof}
We prove this by induction on $b-a$. If $b-a=1$ the proposition is clear because then $c = \widetilde{c} = \delta_i$. So assume $b-a>1$ and the result is known for smaller values of~$b-a$. If $i=a$ or $i=b$ the proposition is also clear because then $c = \widetilde{c}$. So assume that~$a < i < b$. Set $c' \coloneqq \delta_{[a+1,i-1]} + \delta_{i-1}$ and $c'' \coloneqq \delta_{i+1} + \delta_{[i+1,b-1]}$. By firing at vertex~$i$ we see that~$c \to c' + c''$. Applying the inductive hypothesis gives
\begin{align*}
\widetilde{c'} &= \delta_a + \delta_{[a+2,i]}, \\
\widetilde{c''} &= \delta_{[i,b-2]} + \delta_b.
\end{align*}
So $c \to \delta_a + \delta_{[a+2,i]} + \delta_{[i,b-2]} + \delta_b = \delta_a + c''' + \delta_b$ where $c''' \coloneqq \delta_{[a+2,b-2]} + \delta_i$. Applying the inductive hypothesis again gives
\[ \widetilde{c'''} = \delta_{[a+1,a+b-i-1]} + \delta_{[a+b-i+1,b-1]}, \]
so that $c \to \delta_{[a,a+b-i-1]} + \delta_{[a+b-i+1,b]}$. But $\delta_{[a,a+b-i-1]} + \delta_{[a+b-i+1,b]}$ is stable, which means we must have~$\widetilde{c} = \delta_{[a,a+b-i-1]} + \delta_{[a+b-i+1,b]}$.
\end{proof}

The unlabeled configuration we are most interested in is $n\delta_0$. The following description of the stabilization of $n\delta_0$ is also well known, appearing for instance in the original paper of Anderson et al.~\cite{anderson1989disks}. For completeness we provide a short proof of this lemma using the previous proposition.

\begin{lemma} \label{lem:deltastab}
For all $n \geq 1$ we have,
\[\widetilde{n\delta_0} = \begin{cases} \delta_{[-m,-1]} + \delta_{[1, m]} &\textrm{if $n=2m$ is even};\\
\delta_{[-m, m]} &\textrm{if $n=2m+1$ is odd}.\end{cases}\]
\end{lemma}

\begin{proof}
The proof is by induction on $n$. The case $n=1$ is clear; so suppose~$n>1$ and the result is known for $n-1$. Set $c\coloneqq\widetilde{(n\hspace{-0.1cm}-\hspace{-0.1cm}1)\delta_0}$. We have~$\widetilde{n\delta_0} = \widetilde{c+\delta_0}$. If $n=2m$ is even then by induction~$\widetilde{c} = \delta_{[-(m-1),m-1]}$, so~$\widetilde{c+\delta_0} =  \delta_{[-m,-1]} + \delta_{[1, m]}$ by  Proposition~\ref{prop:linestab}. If $n = 2m+1$ is odd, then by induction we have~$\widetilde{c} = \delta_{[-m,-1]} + \delta_{[1,m]}$ and so clearly we have~$\widetilde{c+\delta_0} = c = [-m,m]$.
\end{proof}

Not only is it the case that for any configuration $c$, its stabilization $\widetilde{c}$ is unique, but also, all stabilization sequences $c \to \widetilde{c}$ have the same total number of vertex firings; and what is more, the number of times any given vertex $j \in \mathbb{Z}$ fires in a stabilization sequence $c \to \widetilde{c}$ is also determined. Indeed, it is easy to see from Proposition~\ref{prop:statistics} that the total number of vertex firings in a stabilization sequence $c \to \widetilde{c}$ is $\frac{1}{2}(\varphi^{2}_{\infty}(\widetilde{c}) - \varphi^{2}_{\infty}(c))$, and the number of times $j \in \mathbb{Z}$ fires is $\varphi_{j+1}(c)-\varphi_{j+1}(\widetilde{c})$. Again, for the infinite path, these facts were more or less established in~\cite{anderson1989disks}. These facts continue to hold for chip-firing on arbitrary graphs, as was first established in~\cite{bjorner1991chip}. More generally, that the ``run time'' and ``local run times'' do not depend on the particular way the system evolves is true for all abelian networks; see~\cite[\S2]{bond2016abelian1}. In the following proposition we record these run times and local run times for the initial configuration $n\delta_0$ we are interested in. (We will not use these run times or local run times in what follows, except to point out that the labeled chip-firing sorting algorithm takes cubic time and is thus highly infeasible.)

\begin{prop} \label{prop:runtime}
Let $n \geq 1$ and set $m \coloneqq \lfloor n/2\rfloor$. Then in any stabilization sequence $n\delta_0 \to \widetilde{n\delta_0}$, the total number of vertex firings is $m(m+1)(2m+1)/6$ and the number of times that vertex $j \in \mathbb{Z}$ fires is $(m+1-|j|)(m-|j|)/2$ if $|j| < m$ and $0$ otherwise.
\end{prop}
\begin{proof}
As just mentioned, it follows from Proposition~\ref{prop:statistics} that the total number of vertex firings in a stabilization sequence $c \to \widetilde{c}$ is $\frac{1}{2}(\varphi^{2}_{\infty}(\widetilde{c}) - \varphi^{2}_{\infty}(c))$, and the number of times that~$j \in \mathbb{Z}$ fires is $\varphi_{j+1}(c)-\varphi_{j+1}(\widetilde{c})$. Thus the proposition follows from the description of $\widetilde{n\delta_0}$ in Lemma~\ref{lem:deltastab}.
\end{proof}

Now let us describe labeled chip-firing. A \emph{labeled configuration of chips} on~$\mathbb{Z}$ is some assignment of a finite number of distinguishable chips, labeled by positive integers, to the vertices of~$\mathbb{Z}$. We use uppercase calligraphic script for labeled configurations and use $(i)$ to denote the chip labeled $i$. Formally, we treat a labeled configuration~$\mathcal{C}$ as a function~$\mathcal{C}\colon X \to \mathbb{Z}$ for some $X \subseteq \mathbb{Z}_{>0}$, and we think of chip~$(i)$ as being at the vertex~$\mathcal{C}(i)$ in $\mathcal{C}$ for all~$i \in X$. Normally we will take $X = [n]$ and thus study labeled configurations of the $n$ chips~$(1),(2),\ldots,(n)$.  If $a < b$ and chips $(a)$ and $(b)$ are at the same vertex in $\mathcal{C}$, we may \emph{fire}~$(a)$ and~$(b)$ together in $\mathcal{C}$ by moving~$(a)$ leftward one vertex and~$(b)$ rightward one vertex. (The important point is that {\bf chips with lesser labels move leftward}.)  We write~$\mathcal{C} \to \mathcal{D}$ to mean that $\mathcal{D}$ is obtained from~$\mathcal{C}$ by a sequence of labeled chip-firing moves of this form. If $\mathcal{C}$ is a labeled configuration we use~$[\mathcal{C}]$ to denote the underlying unlabeled configuration: thus $[\mathcal{C}](i) \coloneqq \#\mathcal{C}^{-1}(i)$ for all~$i \in \mathbb{Z}$.  We say that~$\mathcal{D}$ is \emph{stable} if~$[\mathcal{D}]$ is stable. As mentioned, our strategy in understanding labeled chip-firing will be to relate it to unlabeled chip-firing. To that end, here are some very basic facts relating labeled and unlabeled chip-firing, which we will use without even citing specifically from now on.

\begin{prop}~\,
\begin{itemize}
\item If $\mathcal{C} \to \mathcal{D}$ then $[\mathcal{C}] \to [\mathcal{D}]$.
\item If $c \to d$ and $c = [\mathcal{C}]$, then there exists $\mathcal{D}$ with $\mathcal{C} \to \mathcal{D}$ such that $d = [\mathcal{D}]$.
\end{itemize}
Consequently, for any $\mathcal{C}$ there is some stable $\mathcal{D}$ with $\mathcal{C} \to \mathcal{D}$, and we have $[\mathcal{D}] = \widetilde{[\mathcal{C}]}$.
\end{prop}

There need not be a unique stable $\mathcal{D}$ with $\mathcal{C} \to \mathcal{D}$: the previous proposition only determines~$[\mathcal{D}]$ but not the way that the chips are labeled in~$\mathcal{D}$. Nevertheless we are interested in cases where we do have a unique labeled stabilization. In particular, we will consider the labeled analog of $n\delta_0$, which has chips~$(1),(2),\ldots,(n)$ at vertex $0$ and no other chips; we denote this configuration by $\Delta^n$. In other words, $\Delta^n(i) \coloneqq 0$ for all~$i \in [n]$. Of course,~$[\Delta^n] = n\delta_0$. Note, as mentioned in Section~\ref{sec:intro}, that $\Delta^3$ already does not have a unique stabilization. On the other hand, our main result is that when~$n$ is even, $\Delta^n$ does have a unique stabilization.

First let us observe that there is a useful global symmetry in this labeled chip-firing process when we start from the configuration~$\Delta^n$. If $\mathcal{C}$ is a configuration of $n$ labeled chips, define its \emph{dual}~$\mathcal{C}^{*}$ as follows: first reflect $\mathcal{C}$ horizontally about the origin, then replace chip $(i)$ by chip $(n+1-i)$ for all $1 \leq i \leq n$. Of course $(\mathcal{C}^{*})^{*} = \mathcal{C}$.

\begin{lemma} \label{lem:symmetry}
We have $\Delta^n \to \mathcal{C}$ if and only if $\Delta^n \to \mathcal{C}^{*}$.
\end{lemma}
\begin{proof}
It is easy to see that the duality operation respects labeled chip-firing moves, meaning that if $\mathcal{D}$ is obtained from $\mathcal{C}$ by a labeled chip-firing move then $\mathcal{D}^{*}$ is obtained from $\mathcal{C}^{*}$ by a labeled chip-firing move. The lemma then follows since~$(\Delta^n)^* = \Delta^n$.
\end{proof}

Very roughly speaking, to prove confluence of the labeled chip-firing process we study how far we can move chips via chip-firing. The following is obvious but important. 

\begin{prop} \label{prop:stabmaxineq}
If $c \to d$ then $\mathrm{min}(\widetilde{c}) \leq \mathrm{min}(d)$.
\end{prop}
\begin{proof}
Each chip-firing move preserves or decreases the minimum occupied vertex, so we have $\mathrm{min}(d') \leq \mathrm{min}(d)$ for any $d \to d'$. Thus in particular we have $\mathrm{min}(\widetilde{d}) \leq \mathrm{min}(d)$. But if $c \to d$, then $\widetilde{d} = \widetilde{c}$.
\end{proof}

Applying Proposition~\ref{prop:stabmaxineq} to our situation of interest tells us that if~$\Delta^n \to \mathcal{C}$ then we have~$\mathrm{min}([\mathcal{C}]) \geq  -\lfloor n/2\rfloor$ and, by Lemma~\ref{lem:symmetry}, $\mathrm{max}([\mathcal{C}]) \leq \lfloor n/2\rfloor$. This puts some constraint on the movement of chips during the labeled chip-firing process, but is not really so useful because it says nothing about the position of chips with particular labels. We want to strengthen this conclusion about how far chips can move to take into account chip labels. 

Let us establish some notation for restricting labeled configurations to a subset of chips. For a labeled configuration $\mathcal{C}$ with label set~$X$ and $Y \subseteq \mathbb{Z}_{>0}$, we use~$\mathcal{C} \setminus Y$ to denote the restriction of $\mathcal{C}$ to the chips with labels in~$X \setminus Y$. For any labeled configuration~$\mathcal{C}$ and any~$k \in \mathbb{N}$, we use the shorthand $\mathcal{C}|_{\geq k} \coloneqq \mathcal{C} \setminus [k-1]$. We want some way to describe how the largest-labeled chips evolve in the labeled chip-firing process. So let us say that an unlabeled configuration $d$ is \emph{rightward-reachable} from an unlabeled configuration~$c$, written~$c \xrightarrow{R} d$, if $d$ is obtained from~$c$ by a sequence of (zero or more) moves of the forms:
\begin{itemize}
\item perform a chip-firing move;
\item move one chip rightward one vertex.
\end{itemize}
This notion precisely captures the way the largest-labeled chips evolve under labeled chip-firing.  Namely, we have the following.

\begin{prop} \label{prop:labeledrightreach}
If $\mathcal{C} \to \mathcal{D}$ then $[\mathcal{C}|_{\geq k}] \xrightarrow{R} [\mathcal{D}|_{\geq k}]$.
\end{prop}
\begin{proof}
Suppose we fire two chips~$(a)$ and~$(b)$ in $\mathcal{C}$: if $a,b < k$, that firing does not affect~$[\mathcal{C}|_{\geq k}]$; if $k \leq a,b$, that firing  corresponds to a firing in $[\mathcal{C}|_{\geq k}]$; and if $a < k \leq b$, then that firing corresponds to moving a chip rightward in  $[\mathcal{C}|_{\geq k}]$.
\end{proof}

We want a strengthening of Proposition~\ref{prop:stabmaxineq} that applies to rightward-reachability. To that end, we define a partial order on unlabeled configurations of $n$ chips that can informally be thought of as ``$c \leq d$ means $d$ is obtained from $c$ by moving chips rightward''; it is defined formally as follows. If $c$ and $d$ are configurations of~$n$ unlabeled chips on~$\mathbb{Z}$, we write $c \leq d$ if and only if~$\sum_{i \leq j}c(i) \leq \sum_{i \leq j}d(i)$ for all~$j \in \mathbb{Z}$. Observe that $c \leq d$ implies that $\mathrm{max}(c) \leq \mathrm{max}(d)$ and $\mathrm{min}(c) \leq \mathrm{min}(d)$. We write $c \lessdot d$ to mean that $d$ covers $c$ according to this partial order~$\leq$. In other words, $c \lessdot d$ means that $d$ is obtained from $c$ by moving one chip rightward one vertex. 

An important property of this partial order is that it is preserved under stabilization, as we establish right now. In fact, something even stronger is true: stabilization preserves the cover relations of this partial order. (Note that $\varphi_{\infty}$ is a rank function for~$\leq$, where $\varphi_{\infty}(c) \coloneqq \sum_{i \in \mathbb{Z}} i \cdot c(i)$ is the statistic defined earlier in this section. By Proposition~\ref{prop:statistics}, chip-firing moves preserve~$\varphi_{\infty}$. So in fact stabilization being order-preserving is easily seen to be equivalent to it preserving cover relations.)

\begin{lemma} \label{lem:stabcover}
If $c \lessdot d$ then $\widetilde{c} \lessdot \widetilde{d}$.
\end{lemma}
\begin{proof}
That $c \lessdot d$ means there is some $c'$ and $i \in \mathbb{Z}$ such that $c = c' + \delta_i$ and $d = c' + \delta_{i+1}$. Define~$a \coloneqq \mathrm{max}\{j \leq i\colon j \notin \mathrm{supp}(\widetilde{c'})\}$ and $b \coloneqq \mathrm{min}\{j \geq i+1\colon j \notin \mathrm{supp}(\widetilde{c'})\}$. Thus there exists a configuration $c''$ such that $ \widetilde{c'} = c'' + \delta_{[a+1,b-1]}$ and $\mathrm{supp}(c'') \cap [a,b] = \varnothing$. Proposition~\ref{prop:linestab} then implies
\begin{align*}
\widetilde{\widetilde{c'} + \delta_i} &= c'' +\delta_{[a,a+b-i-1]} + \delta_{[a+b-i+1,b]}, \\
\widetilde{\widetilde{c'} \hspace{-0.1cm} + \hspace{-0.1cm} \delta_{i+1}} &= c'' +\delta_{[a,a+b-i-2]} + \delta_{[a+b-i,b]}. 
\end{align*}
In particular,
$\widetilde{\widetilde{c'} + \delta_i} \lessdot\widetilde{\widetilde{c'} \hspace{-0.1cm} + \hspace{-0.1cm} \delta_{i+1}} $. But $c = \widetilde{\widetilde{c'} + \delta_i}$ and $d=\widetilde{\widetilde{c'} \hspace{-0.1cm} + \hspace{-0.1cm} \delta_{i+1}}$, so the claim is proved.
\end{proof}

Lemma~\ref{lem:stabcover} is the key lemma which allows us to establish confluence of labeled chip-firing. It is also interesting in its own right as a result purely concerning unlabeled chip-firing. We now apply Lemma~\ref{lem:stabcover} to give a strengthening of Proposition~\ref{prop:stabmaxineq} which applies to rightward-reachability.

\begin{cor} \label{cor:rightreach}
If $c \xrightarrow{R} d$ then $\widetilde{c} \leq \widetilde{d}$ and consequently $\mathrm{min}(\widetilde{c}) \leq \mathrm{min}(d)$.
\end{cor}

\begin{proof}
Suppose $c \xrightarrow{R} d$. Thus there is some sequence~$c_0,c'_0,c_1,c'_1,\ldots,c_{\ell},c'_{\ell}$ of configurations with $c = c_0$ and $c'_{\ell} = d$ such that:
\begin{itemize}
\item $c_i \to c'_i$ for all $0 \leq i \leq \ell$;
\item $c'_{i-1} \lessdot c_i$ for all $1 \leq i \leq \ell$.
\end{itemize}
We claim that $ \widetilde{c} \leq \widetilde{c_i} = \widetilde{c'_i}$ for all $0 \leq i \leq \ell$. That $\widetilde{c_i} = \widetilde{c'_i}$ follows from $c_{i} \to c'_{i}$. So the crucial part of the claim is to show $\widetilde{c}\leq\widetilde{c_i}$. Clearly this holds for $i=0$ since by definition~$c_0 = c$. So assume $1 \leq i \leq \ell$ and $ \widetilde{c}\leq \widetilde{c_{i-1}}$. Because~$c'_{i-1} \lessdot c_i $, from Lemma~\ref{lem:stabcover} we get that~$ \widetilde{c_{i-1}} = \widetilde{c'_{i-1}} \lessdot \widetilde{c_i}$. Together with~$\widetilde{c} \leq \widetilde{c_{i-1}}$ this implies~$\widetilde{c} \leq \widetilde{c_i}$. So the claim is proved by induction.  Taking $i=\ell$ in the claim gives~$\widetilde{c} \leq \widetilde{c'_{\ell}} $, which is to say~$\widetilde{c} \leq \widetilde{d}$. This implies $\mathrm{min}(\widetilde{c}) \leq \mathrm{min}(\widetilde{d})$. But $\mathrm{min}(\widetilde{d}) \leq \mathrm{min}(d)$ by Proposition~\ref{prop:stabmaxineq}.
\end{proof}

Now we can apply Corollary~\ref{cor:rightreach} to restrict, based on their labels, how far chips can move in our situation of interest.

\begin{lemma} \label{lem:labeledchippos}
Suppose $\Delta^n \to \mathcal{C}$. Then~$-\lfloor (n+1-k)/2 \rfloor \leq \mathcal{C}(k) \leq \lfloor k/2 \rfloor$ for all~$1 \leq k \leq n$.
\end{lemma}
\begin{proof}
First we show $-\lfloor (n+1-k)/2 \rfloor \leq \mathcal{C}(k)$. By Proposition~\ref{prop:labeledrightreach}, $[\Delta^n|_{\geq k}] \xrightarrow{R} [C|_{\geq k}]$. Thus by Lemma~\ref{cor:rightreach}, $\mathrm{min}(\widetilde{[\Delta^n|_{\geq k}]}) \leq \mathrm{min}([C|_{\geq k}])$. But $[\Delta^n|_{\geq k}] = (n+1-k)\delta_0$, and so Lemma~\ref{lem:deltastab} tells us that~$\mathrm{min}(\widetilde{[\Delta^n|_{\geq k}]}) = -\lfloor (n+1-k)/2\rfloor$. Thus indeed chip $(k)$ must be at or to the left of the vertex~$-\lfloor (n+1-k)/2\rfloor$. That $\mathcal{C}(k) \leq \lfloor k/2 \rfloor$ then follows via Lemma~\ref{lem:symmetry}.
\end{proof}

We are now ready to prove the main theorem, which says that when the number~$n$ of chips is even, the labeled chip-firing process on~$\mathbb{Z}$ necessarily sorts these chips. Recall that, according to Proposition~\ref{prop:runtime}, the number of firings in this process is $\Theta(n^3)$, so this procedure is not being offered as a practical way to sort.

\begin{thm} \label{thm:main}
Suppose $n \coloneqq 2m$ is even and $\Delta^n \to \mathcal{D}$ where $\mathcal{D}$ is stable. Then for all~$1 \leq k \leq m$ we have that $\mathcal{D}(k) = -(m+1)+k$ and $\mathcal{D}(m+k) = k$.
\end{thm}

\begin{proof} 
Let $n =2m$ be even and let $\Delta^n \to \mathcal{D}$ with $\mathcal{D}$ stable. For all $1 \leq k \leq m$, the assertion that $\mathcal{D}(m+k) = k$ follows from $\mathcal{D}(m+1-k) = -k$ by Lemma~\ref{lem:symmetry}. Thus we prove only that~$\mathcal{D}(k) = -(m+1)+k$ for all $1 \leq k \leq m$.

The proof is by induction on $k$. So let us first address the base case $k=1$. Lemma~\ref{lem:labeledchippos} says that $\mathcal{D}(i) > -m$ for all $2 \leq i \leq n$. (Here we use crucially that $n=2m$ is even.) But on the other hand, we know thanks to Lemma~\ref{lem:deltastab} that vertex $-m$ is occupied in $\mathcal{D}$. So in fact it must be occupied by chip~(1).

Now assume $k \geq 2$ and the result holds for all smaller values of~$k$. We will use some internal lemmas in the proof (``internal'' because they assume the inductive hypothesis).

\begin{lemma} \label{lem:prooflem1}
If $\mathcal{D}(k) > -(m+1)+k$ then for all~$1 \leq j \leq k-1$, chip $(k)$ never fired together with chip~$(j)$ in the labeled chip-firing process $\Delta^n \to \mathcal{D}$.
\end{lemma}
\begin{proof}
Suppose that $\mathcal{D}(k) > -(m+1)+k$. And suppose to the contrary that chip~$(k)$ did fire together with chip~$(j)$ for some $1 \leq j \leq k-1$ at some point in the labeled chip-firing process $\Delta^n \to \mathcal{D}$. Let us concentrate on the last moment when this happened: let~$\mathcal{C}'$ be the step before chip~$(k)$ fired with some chip~$(j)$ with $1 \leq j \leq k-1$ for the last time (and thus define $j$ to be the label of this other chip). Let $\mathcal{C}$ be the result of firing~$(k)$ and~$(j)$ together in $\mathcal{C}'$. So $\Delta^n \to \mathcal{C}'$, $\mathcal{C}$ is obtained from $\mathcal{C}'$ by firing~$(k)$ and~$(j)$ together, and $\mathcal{D}$ is obtained from $\mathcal{C}$ by a sequence of firings that either do not involve~$(k)$, or fire~$(k)$ together with a chip with a greater label. It is clear from this description that~$[\mathcal{C} \setminus \{k\}] \xrightarrow{R} [\mathcal{D} \setminus \{k\}]$. Therefore, Corollary~\ref{cor:rightreach} tells us that~$\widetilde{[\mathcal{C} \setminus \{k\}]} \leq [\mathcal{D} \setminus \{k\}]$. As a consequence of the  assumption~$\mathcal{D}(k) > -(m+1)+k$, the assumption $k \leq m$, and Lemma~\ref{lem:deltastab}, we have that~$[-m,-(m+1)+k] \subseteq \mathrm{supp}(\widetilde{[\mathcal{D} \setminus \{k\}]})$. Thus, since $\mathrm{min}(\widetilde{[\mathcal{C} \setminus \{k\}]}) \geq \mathrm{min}(\widetilde{n\delta_0})= -m$ and $\widetilde{[\mathcal{C} \setminus \{k\}]}$ has at most one chip at each vertex, we have~$[-m,-(m+1)+k] \subseteq \mathrm{supp}(\widetilde{[\mathcal{C} \setminus \{k\}]})$. Next, note~$[\mathcal{C}\setminus \{k\}] \lessdot [\mathcal{C}' \setminus\{k\}]$. So by applying Lemma~\ref{lem:stabcover}, we conclude that $\widetilde{[\mathcal{C}\setminus \{k\}]} \lessdot \widetilde{[\mathcal{C}' \setminus\{k\}]}$, i.e., that $ \widetilde{[\mathcal{C}' \setminus\{k\}]}$ is obtained from  $\widetilde{[\mathcal{C}\setminus \{k\}]}$ by moving one chip rightward one vertex. In particular this means that we must have~$[-m,-(m+1)+k-1] \subseteq \mathrm{supp}(\widetilde{[\mathcal{C}' \setminus \{k\}]})$ (where again we use the fact that that~$\widetilde{[\mathcal{C}' \setminus \{k\}]}$ has at most one chip at each vertex). Now, chips~$(k)$ and~$(j)$ occupy the same vertex in~$\mathcal{C}'$, which means $[\mathcal{C}' \setminus \{k\}] = [\mathcal{C}' \setminus \{j\}]$. So by starting from $\mathcal{C}'$ and repeatedly firing all chips other than~$(j)$ until we stabilize these other chips, we can eventually reach some configuration~$\mathcal{D}'$ with $[\mathcal{D}' \setminus \{j\}] = \widetilde{[\mathcal{C}' \setminus \{j\}]} = \widetilde{[\mathcal{C}' \setminus \{k\}]}$. 

The upshot of the previous paragraph is that if the lemma is false then we can find a configuration~$\mathcal{D}'$ with~$\Delta^n \to \mathcal{D}'$ and~$[-m,-(m+1)+k-1] \subseteq \mathrm{supp}([\mathcal{D}' \setminus \{j\}])$ for some~$1 \leq j \leq k-1$. Let us show that this is impossible. For an unlabeled configuration~$c$ and $\ell \in \mathbb{Z}$, recall the statistic $\varphi_{\ell}(c)  \coloneqq \sum_{i \leq \ell} (i - \ell-1)\cdot c(i)$ defined at the beginning of this section. It follows from Proposition~\ref{prop:statistics} that $\varphi_{\ell}$ weakly decreases with each chip-firing move, and so we always have $\varphi_{\ell}(\widetilde{c}) \leq \varphi_{\ell}(c)$; moreover, it follows from Proposition~\ref{prop:statistics} that if~$\varphi_{\ell}(c) = \varphi_{\ell}(\widetilde{c})$ then vertex $\ell+1$ never fires during the stabilization process~$c \to \widetilde{c}$. Now, we claim that~$(j)$ is strictly to the right of vertex $-(m+1)+k-1$ in~$\mathcal{D}'$: indeed, otherwise~$\varphi_{-(m+1)+k-1}([\mathcal{D}']) < \varphi_{-(m+1)+k-1}(\widetilde{n\delta_0})$, and of course~$\widetilde{[\mathcal{D}']} = \widetilde{n\delta_0}$. If chip~$(j)$ is strictly to the right of vertex $-(m+1)+k-1$ in~$\mathcal{D}'$, as it must be, then~$\varphi_{-(m+1)+k-1}([\mathcal{D}']) = \varphi_{-(m+1)+k-1}(\widetilde{n\delta_0})$. So if we continue to stabilize, that is, if we let $\mathcal{D}''$ be such that $\mathcal{D}' \to \mathcal{D}''$ and $\mathcal{D}''$ is stable, then the vertex $-(m+1)+k$ never fires during the labeled chip-firing process~$\mathcal{D}' \to \mathcal{D}''$. Consequently, chip~$(j)$ always remains strictly to the right of~$-(m+1)+k-1$ during the process $\mathcal{D}' \to \mathcal{D}''$. So chip~$(j)$ is strictly to the right of~$-(m+1)+k-1$ in the stable configuration $\mathcal{D}''$. But this contradicts our inductive hypothesis since $1 \leq j \leq k-1$.
\end{proof}

\begin{lemma} \label{lem:prooflem2}
Chip $(k)$ must have fired together with chip~$(k-1)$ at some point in the labeled chip-firing process $\Delta^n \to \mathcal{D}$.
\end{lemma}
\begin{proof}
Note that in the labeled chip-firing process, chips~$(k)$ and~$(k-1)$ interact in the same way with all chips~$(j)$ for~$j \neq k,k-1$. So if chip~$(k)$ and chip~$(k-1)$ never fire together in the labeled chip-firing process $\Delta^n \to \mathcal{D}$, we can swap the roles of~$(k)$ and~$(k-1)$ to reach a stable configuration $\mathcal{D}'$ where $(k)$ and $(k-1)$ have swapped places. This contradicts our inductive hypothesis which says that there is only one vertex~$(k-1)$ could end up at in a stable configuration.
\end{proof}

Lemmas~\ref{lem:prooflem1} and~\ref{lem:prooflem2} together imply that $\mathcal{D}(k) \leq -(m+1)+k$. By our inductive hypothesis, we know that vertex $-(m+1)+j$ is occupied by $(j)$ for all $1 \leq j \leq k-1$. Thus~$\mathcal{D}(k) = -(m+1)+k$. Therefore, the theorem is proved by induction.
\end{proof}

\begin{remark} 
Caroline Klivans pointed out the following to us. For $\mathcal{C}$ a labeled configuration and~$k \in \mathbb{Z}_{>0}$, define
\[ \psi_{k}(\mathcal{C}) \coloneqq \sum_{1\leq \ell \leq k} \mathcal{C}(\ell).\]
Suppose $\mathcal{C}'$ is obtained from $\mathcal{C}$ by firing chips~$(i)$ and $(j)$ with $i < j$; then it is easy to see that
\[  \psi_{k}(\mathcal{C}') = \begin{cases} \psi_{k}(\mathcal{C}) - 1 &\textrm{if $i \leq k$ and $j > k$}; \\ \psi_{\ell}(\mathcal{C}) &\textrm{otherwise}. \end{cases}\]
Thus for any $k \in \mathbb{Z}_{>0}$, in any labeled stabilization process $\mathcal{C} \to \mathcal{D}$ the number of times that a chip~$(i)$ with $i \leq k$ fired with a chip~$(j)$ where $k < j$ is $\psi_{k}(\mathcal{C}) - \psi_{k}(\mathcal{D})$. So as a consequence of Theorem~\ref{thm:main} we arrive at the following global invariant of the labeled chip-firing process with initial configuration $\Delta^n$ for $n$ even, which can be compared to Proposition~\ref{prop:runtime}.
\begin{cor} \label{cor:main}
Let $n \coloneqq 2m$ be even, and $1 \leq k \leq n$. Then in any labeled stabilization process $\Delta^n \to \mathcal{D}$ the number of times that a chip~$(i)$ with $i \leq k$ fired with a chip~$(j)$ where~$k < j$ is $(m-|k-m|)(m+|k-m|+1)/2$.
\end{cor}
\begin{proof}
As just mentioned, the number of such firings is $\psi_{k}(\Delta^n) - \psi_{k}(\mathcal{D}) = -\psi_{k}(\mathcal{D})$. From Theorem~\ref{thm:main} we can compute that $ -\psi_{k}(\mathcal{D})=(m-|k-m|)(m+|k-m|+1)/2$.
\end{proof}
In fact, Corollary~\ref{cor:main} is easily seen to be equivalent to Theorem~\ref{thm:main}. But we know no simpler reason why Corollary~\ref{cor:main} should be true, beyond the proof of Theorem~\ref{thm:main} we have given above.
\end{remark}

\section{Extensions} \label{sec:extensions}

\subsection{Other graphs}

An obvious question is if the labeled chip-firing process can be extended to other graphs beyond~$\mathbb{Z}$. Ideally any such extension would have unique labeled stabilizations for many of its initial configurations. While we are far from being able to propose an interesting extension of labeled chip-firing to arbitrary graphs, we have found that several minor variants of the infinite path (apparently) continue to exhibit confluence of certain initial configurations. 

Let~$G$ be a directed graph with vertex set~$V$. We allow parallel edges and loops. There is a well-known notion of unlabeled chip-firing on $G$ (see e.g.~\cite[\S6]{corry2017divisors}). Briefly, we study the evolution of a configuration $c$ of $n$ indistinguishable chips on $V$ under the following chip-firing moves: we may fire a vertex $v \in V$ as long as it has as many chips as its outdegree~$\mathrm{outdeg}_G(v)$; firing at $v$ transports one chip from $v$ to $u$ along each outgoing edge $(v,u)$ from $v$. (Transporting a chip along a loop $(v,v)$ keeps that chip at~$v$, but loops do affect the dynamics of the system because they force there to be more chips at a vertex before it can fire.) In this context, we consider an undirected graph to be a directed graph where each undirected edge $\{u,v\}$ corresponds to two directed edges~$(u,v)$ and $(v,u)$; however, we will always treat loops $(v,v)$ as directed loops even on otherwise undirected graphs. In the examples that follow we will always have $V=\mathbb{Z}$ or~$\mathbb{N}$ and we will draw these graphs in the plane as a number line in the usual way. Consequently, we can consider $\delta_{[a,b]}$  and $\delta_i$ to be configurations on~$G$ for appropriate values of $a,b,i$. For clarity we say things like ``$c$ is $G$-stable'' to mean that~$c$ is stable when considered as a configuration on $G$. We write~$c \xrightarrow{G} d$ to mean that~$d$ is obtained from $c$ by a series of $G$-chip-firing moves, and we use $\widetilde{c}^{\, G}$ to denote the $G$-stabilization of $c$ (which will exist and be unique for all graphs under consideration). Now let us describe one framework for labeled chip-firing on $G$. A labeled configuration~$\mathcal{C}$ on~$G$ is some assignment of a finite number of distinguishable chips, labeled by positive integers, to~$V$. Suppose that each vertex $v \in V$ has been given a total order~$e_1 < e_2 < \ldots < e_{\mathrm{outdeg}_G(v)}$ on its outgoing edges. Then a labeled chip-firing move at $v$ consists of choosing $\mathrm{outdeg}_G(v)$ chips $(i_1), (i_2), \ldots, (i_{\mathrm{outdeg}_G(v)})$ which all occupy~$v$, with~$i_1 < i_2 < \cdots < i_{\mathrm{outdeg}_G(v)}$, and transporting $(i_1)$ along $e_1$, $(i_2)$ along~$e_2$, et cetera. The graphs we consider here all have the same local structure: at any vertex~$v \in V$, there are $\ell_v$ directed edges from~$v$ to the vertex immediately to its left, $m_v$ directed loops at~$v$, and $r_v$ directed edges from~$v$ to the vertex immediately to its right; i.e., each $v \in V$ looks like the following:
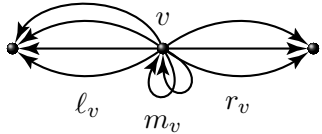
\begin{center}
\begin{tikzpicture}[scale=1]
	\SetFancyGraph
	\Vertex[NoLabel,x=-2,y=0]{-1}
	\Vertex[LabelOut,Lpos=90, Ldist=.1cm,x=0,y=0]{v}
	\Vertex[NoLabel,x=2,y=0]{1}
	\Edges[style={thick,->,bend right=65}](v,-1)
	\Edges[style={thick,->,bend right=35}](v,-1)
	\Edges[style={thick,->,bend right=0}](v,-1)
	\Edges[style={thick,->,bend right=-35}](v,-1)
	\Edges[style={thick,->,out=330,in=270,looseness=30}](v,v)
	\Edges[style={thick,->,out=300,in=240,looseness=30}](v,v)
	\Edges[style={thick,->,bend right=35}](v,1)
	\Edges[style={thick,->,bend right=0}](v,1)
	\Edges[style={thick,->,bend right=-35}](v,1)
	\node at (-1,-0.75) {$\ell_v$};
	\node at (0,-1) {$m_v$};
	\node at (1,-0.75) {$r_v$};
\end{tikzpicture}
\end{center}
The order we give to the outgoing edges at $v$ will always be: all the left edges are less than the loops, which in turn are less than the right edges. Thus, firing at $v$ consists of choosing $\ell_v + m_v + r_v$ chips occupying $v$, moving the $\ell_v$ of them with the smallest labels to the left, the $r_v$ of them with the largest labels to the right, and keeping the~$m_v$ ``middle'' chips at $v$. We use notation for labeled configruations on $G$ in a predictable way: we write $[\mathcal{C}]$ to denote the underlying unlabeled configuration of $\mathcal{C}$; we say $\mathcal{C}$ is $G$-stable if $[\mathcal{C}]$ is $G$-stable; we write $\mathcal{C} \xrightarrow{G} \mathcal{D}$ to mean that $\mathcal{D}$ is obtained from~$\mathcal{C}$ by a series of labeled chip-firing moves on $G$. Since the origin will belong to~$V$ in all graphs we consider, we can still consider $\Delta^n$ to be a labeled configuration on~$G$. We are interested in confluence, of course, so let us say that $G$ \emph{sorts} $\Delta^n$ if there is a unique $G$-stable $\mathcal{D}$ with $\Delta^n \xrightarrow{G} \mathcal{D}$, and for all $1 \leq i \leq j \leq n$, we have $\mathcal{D}(i) \leq \mathcal{D}(j)$.

First let us describe some ``one-way infinite'' paths for which it is quite easy to see that sorting always occurs. 

\begin{prop} \label{prop:nundirected}
Let $G$ be the undirected graph with vertex set $\mathbb{N}$ and with a single edge $\{i,i+1\}$ for each $i \in \mathbb{N}$. Then $G$ sorts $\Delta^n$ for any $n \geq 1$.
\end{prop}
\begin{proof}
It is easy to check inductively that $\widetilde{n\delta_0}^{G} = \delta_{[1,n]}$. Then we can apply a very simple ``reachability'' argument. Note that chip~$(1)$ can never make it to vertex~$2$ (or any vertex right of~$2$) because to do so it would have to fire with a smaller-labeled chip at vertex~$1$. Similarly, chip~$(2)$ can never make it to vertex $3$ because it would have to fire with a smaller-labeled chip at vertex~$2$, and $(1)$ will never be at vertex~$2$. And thus~$(3)$ can never make it to vertex~$4$, and so on.
\end{proof}

\begin{prop} \label{prop:ndirected}
Let $G$ be the directed graph with vertex set $\mathbb{N}$ which has a single directed edge $(i,i+1)$ and a loop $(i,i)$ for each $i \in \mathbb{N}$. Then $G$ sorts $\Delta^n$ for any $n \geq 1$.
\end{prop}
\begin{proof}
The proof is the same as the proof of Proposition~\ref{prop:nundirected}: we can check inductively that $\widetilde{n\delta_0}^{G} = \delta_{[0,n-1]}$; but chip~$(1)$ can never it make it to vertex~$1$, so chip~$(2)$ can never make it to vertex~$2$, and so on.
\end{proof}

The directed graph appearing in Proposition~\ref{prop:ndirected} is noteworthy because unlike the other undirected graphs we have been studying, it takes only $\Theta(n^2)$ firings to stabilize. Indeed, the labeled chip-firing process on this graph basically caries out ``bubble sort,'' or one of the other well-known $\Theta(n^2)$ sorting algorithms, depending on the order in which we fire the vertices. To our knowledge, this is the best you can do with labeled chip-firing in terms of run time: we know of no graph that sorts in time~$\Theta(n\, \log n)$.

Next, we will consider some variants of the two-way infinite path~$\mathbb{Z}$, specifically, graphs obtained from~$\mathbb{Z}$ by adding loops or parallel edges. Here we mostly offer conjectures, except for the following minor modification of Theorem~\ref{thm:main}.

\begin{thm} \label{thm:loops}
Let $G$ be the graph obtained from the infinite path~$\mathbb{Z}$ by adding $\ell$ loops at the origin. Then $G$ sorts $\Delta^n$ whenever~$n \equiv \ell \mod 2$.
\end{thm}
\begin{proof}
The proof is basically the same as the proof of Theorem~\ref{thm:main} we gave in Section~\ref{sec:main}. Let us sketch how to modify that proof to accommodate loops at the origin. First of all, one can compute inductively that 
\[\widetilde{n\delta_0}^{G} = \delta_{[-\lfloor(n-\ell)/2\rfloor,-1]} + \delta_{[1,\lfloor(n-\ell)/2\rfloor]} + \begin{cases} \ell\delta_0 &\textrm{if $n \equiv \ell \mod 2$}, \\ (\ell+1)\delta_0 &\textrm{otherwise}. \end{cases}\]
It is easy to see that the symmetry lemma, Lemma~\ref{lem:symmetry}, remains true in this context: we have $\Delta^n \xrightarrow{G} \mathcal{C} \Rightarrow \Delta^n \xrightarrow{G} \mathcal{C}^{*}$. It is also easy to see that Proposition~\ref{prop:labeledrightreach} remains true in this context: we have $\mathcal{C} \xrightarrow{G} \mathcal{D} \Rightarrow [\mathcal{C}\mid_{\geq k}] \xrightarrow{R} [\mathcal{D}\mid_{\geq k}]$ for all~$k \geq 1$. Moreover, the key lemma, Lemma~\ref{lem:stabcover}, also remains true: we have $c \lessdot d \Rightarrow \widetilde{c}^{G} \lessdot \widetilde{d}^{G}$; we should modify the proof of Lemma~\ref{lem:stabcover} given above by considering $a := \mathrm{max}\{j \leq i\colon c(j) \leq \mathrm{outdeg}_G(j)-2\}$ and $b := \mathrm{min}\{j \geq i+1\colon c(j) \leq \mathrm{outdeg}_G(j)-2\}$ instead. Altogether, this means we the following generalization of Lemma~\ref{lem:labeledchippos}: suppose $\Delta^n \xrightarrow{G} \mathcal{C}$; then, for all $1 \leq k \leq n$,
\[-\lfloor (n-\ell+1-k)/2 \rfloor \leq \mathcal{C}(k) \leq \lfloor (k-\ell)/2 \rfloor.\] 
With this set-up, how do we prove that sorting occurs? Set $n := 2m+\ell$ and suppose that~$\Delta^n \xrightarrow{G} \mathcal{D}$ where $\mathcal{D}$ is $G$-stable. It suffices to prove that $\mathcal{D}(k) = -(m+1) + k$ for all~$1 \leq k \leq m$. This is because the symmetry lemma will then imply~$\mathcal{D}(k) = k-(m+\ell)$ for all $m+\ell+1 \leq k \leq 2m+\ell$; and then, since we know what $\widetilde{n\delta_0}^{G}$ looks like, we will have to have $\mathcal{D}(k) = 0$ for all $m+1 \leq k \leq m+\ell$. So we prove $\mathcal{D}(k) = -(m+1) + k$ for all $1 \leq k \leq m$ by induction on $k$. The base case $k=1$ follows from the generalization of  Lemma~\ref{lem:labeledchippos} stated above. Thus we assume $k \geq 2$ and the claim holds for smaller values of~$k$. Again, we want to prove analogs of the internal lemmas in the proof of Theorem~\ref{thm:main}. First of all, in both Lemma~\ref{lem:prooflem1} and Lemma~\ref{lem:prooflem2} we should interpret ``chip~$(i)$ fires together with chip~$(j)$'' to include the possibility that $(i)$ and $(j)$ fire together in a group with other chips at the origin as well. With this understood, the statement and proof of Lemma~\ref{lem:prooflem2} goes through exactly. As for Lemma~\ref{lem:prooflem1}, the statement goes through exactly. There is one slight change necessary in the proof, which is that we do not necessarily have $[\mathcal{C} \setminus \{k\}] \lessdot [\mathcal{C}' \setminus \{k\}]$ (because it could be that $(k)$ traveled along a loop during the firing between steps~$\mathcal{C}'$ and $\mathcal{C}$); rather, we have $c \lessdot [\mathcal{C}' \setminus \{k\}]$ for some~$c \xrightarrow{R} [\mathcal{C} \setminus \{k\}]$. This still forces $[-m,-(m+1)+k-1] \subseteq \mathrm{supp}(\widetilde{[\mathcal{C}'\setminus \{k\}]}^{G})$ as required, and the rest of the proof of Lemma~\ref{lem:prooflem1} then goes through in the same way. The analogs of Lemma~\ref{lem:prooflem1} and Lemma~\ref{lem:prooflem2} together immediately imply $\mathcal{D}(k) = -(m+1) + k$, and so the proof is complete by induction.
\end{proof}

For~$S \subseteq \mathbb{Z}$, let us use~$\mathbb{Z}\langle S \rangle$ to denote the graph obtained from the infinite path~$\mathbb{Z}$ by adding a single loop at each~$i \in S$.

\begin{conj} \label{conj:loops}
Let $S \subseteq \mathbb{Z}$. Suppose $n \in \mathbb{N}$ is such that:
\begin{itemize}
\item $\mathrm{min}(\widetilde{n\delta_0}^{\mathbb{Z}\langle S \rangle}) < \mathrm{min}(\widetilde{(n-1)\delta_0}^{\mathbb{Z}\langle S \rangle})$;
\item $\mathrm{max}(\widetilde{n\delta_0}^{\mathbb{Z}\langle S \rangle}) > \mathrm{max}(\widetilde{(n-1)\delta_0}^{\mathbb{Z}\langle S \rangle})$.
\end{itemize}
Then $\mathbb{Z}\langle S \rangle$ sorts $\Delta^n$.
\end{conj}

\begin{remark}
Any two of the following three implies the third:
\begin{itemize}
\item $\mathrm{min}(\widetilde{n\delta_0}^{\mathbb{Z}\langle S \rangle}) < \mathrm{min}(\widetilde{(n-1)\delta_0}^{\mathbb{Z}\langle S \rangle})$;
\item $\mathrm{max}(\widetilde{n\delta_0}^{\mathbb{Z}\langle S \rangle}) > \mathrm{max}(\widetilde{(n-1)\delta_0}^{\mathbb{Z}\langle S \rangle})$;
\item $\displaystyle \sum_{i \in S'} i = 0$ where $S' \coloneqq S \cap [\mathrm{min}(\widetilde{n\delta_0}^{\mathbb{Z}\langle S \rangle})+1,\mathrm{max}(\widetilde{n\delta_0}^{\mathbb{Z}\langle S \rangle})-1]$.
\end{itemize}
Thus, Conjecture~\ref{conj:loops} only really applies to $S$ which are ``balanced'' around the origin. Some special cases of Conjecture~\ref{conj:loops} are worth pointing out specifically. The case $S = \varnothing$ of Conjecture~\ref{conj:loops} is just Theorem~\ref{thm:main}, and the case $S = \{0\}$ follows from Theorem~\ref{thm:loops}. The case $S = \mathbb{Z}$ of Conjecture~\ref{conj:loops} says that $\mathbb{Z}\langle \mathbb{Z} \rangle$ sorts $\Delta^n$ as long as $n \equiv 3 \mod 4$.
\end{remark}

\begin{remark}
It is natural to consider also graphs obtained from~$\mathbb{Z}$ by adding multiple loops at vertices. Indeed, with Theorem~\ref{thm:loops} above we have already seen that some such graphs sort~$\Delta^n$ for infinitely many values of~$n$. However, one has to be careful in trying to generalize too much, because, for instance, the graph obtained from~$\mathbb{Z}$ by adding two loops at every vertex does not sort $\Delta^n$ for any~$n \geq 5$.
\end{remark}

For~$r \geq 1$, we use $r\mathbb{Z}$ to denote the graph obtained from $\mathbb{Z}$ by replacing each edge by $r$ parallel edges.

\begin{conj} \label{conj:parallel}
For all $r \geq 1$, $r\mathbb{Z}$ sorts $\Delta^n$ whenever~$n \equiv 0 \mod 2r$.
\end{conj}

In fact, Conjectures~\ref{conj:loops} and~\ref{conj:parallel} can be simultaneously generalized. Let us use $r(\mathbb{Z}\langle S \rangle)$ to denote the graph obtained from $\mathbb{Z}\langle S \rangle$ by replacing each edge with $r$ parallel edges, including replacing each loop by $r$ loops. 

\begin{conj}
Let $S \subseteq \mathbb{Z}$. Suppose $n \in \mathbb{N}$ is such that:
\begin{itemize}
\item $\mathrm{min}(\widetilde{n\delta_0}^{\mathbb{Z}\langle S \rangle}) < \mathrm{min}(\widetilde{(n-1)\delta_0}^{\mathbb{Z}\langle S \rangle})$;
\item $\mathrm{max}(\widetilde{n\delta_0}^{\mathbb{Z}\langle S \rangle}) > \mathrm{max}(\widetilde{(n-1)\delta_0}^{\mathbb{Z}\langle S \rangle})$.
\end{itemize}
Then $r(\mathbb{Z}\langle S \rangle)$ sorts $\Delta^{rn}$ for each $r \geq 1$.
\end{conj}

\subsection{Other configurations}

We use the notation $\widetilde{\mathcal{C}} \coloneqq \{\mathcal{D}\colon \mathcal{C} \to \mathcal{D} \textrm{ and $\mathcal{D}$ is stable}\}$. Another natural problem is to understand $\widetilde{\mathcal{C}}$ for more configurations~$\mathcal{C}$ on the infinite path~$\mathbb{Z}$. For example, for $n=1,3,5,7,9,\ldots$ we have $\#\widetilde{\Delta^{n}} = 1,3,12,54,232,\ldots$ (a sequence which is unfortunately not in the OEIS~\cite{oeis2016}). We understand the even case, so let us concentrate on this odd case; thus set~$n\coloneqq2m+1$. Since~$[\mathcal{D}] = [-m,m]$ for~$\mathcal{D} \in \widetilde{\Delta^{n}}$, we may identify elements of~$\widetilde{\Delta^{n}}$ with permutations. Completely describing the permutations in~$\widetilde{\Delta^{n}}$ seems hard, but there are at least a few nontrivial things we can say. First of all, Lemma~\ref{lem:labeledchippos} applies equally when $n$ is odd and puts some restrictions on~$\widetilde{\Delta^{n}}$. We can also say the following: for any injective, order-preserving map~$\iota\colon [n] \to [n+1]$, if we relabel a configuration~$\mathcal{D} \in \widetilde{\Delta^{n}}$ according to $\iota$, add a new chip~$(j)$ to the origin where $\{j\} \coloneqq [n+1] \setminus \mathrm{im}(\iota)$, and then stabilize the resulting configuration, the chips have to appear in sorted order. Indeed, this is a consequence of our main theorem, Theorem~\ref{thm:main}, because one possible way to stabilize~$\Delta^{n+1}$ is to ignore chip~$(j)$ for as long as possible and instead first stabilize the chips with labels in~$\mathrm{im}(\iota)$. Even these two conditions (Lemma~\ref{lem:labeledchippos} and the ``add a chip and stabilize to sort'' condition) together fail to completely characterize $\widetilde{\Delta^{n}}$, however, because for instance the permutation $23154$ satisfies both of these conditions but does not belong to $\widetilde{\Delta^{5}}$. We can at least offer the following attractive conjecture about $\widetilde{\Delta^{n}}$, which has been verified for $n \leq 9$ odd.

\begin{conj}
For $n=2m+1$, the maximum number of inversions among all permutation in~$\widetilde{\Delta^{n}}$ is exactly~$m$.
\end{conj}

A different way to understand configurations for which there is not unique labeled stabilization would be probabilistically.  There are at least three reasonable ways to carry out labeled chip-firing randomly: (1) at each step choose a chip-firing move uniformly at random among all possible moves; (2) at each step choose an unstable vertex uniformly at random and then choose a pair of chips at that vertex uniformly at random; or (3) choose a stabilization sequence uniformly at random among all (labeled) stabilization sequences.  Based on some limited computer simulations, it appears that when $m$ is large, random labeled chip-firing applied to $\Delta^{2m+1}$ leads to all chips ending up sorted with probability around $.33$ under all three protocols. We have no intuition for why this probability should not converge to~$0$. It is clear that it cannot converge to a limit greater than $1/3$, since $2/3$ of the time the ``last move'' fails to put the chips in sorted order. (It is not hard to see that the last move necessarily involves firing a vertex that has three chips on it, and so only one of the three possible labeled firings of these three chips will locally sort them.)

\begin{conj} \label{conj:random}
With respect to any of the three protocols for random labeled chip-firing described above, the probability that $\Delta^{2m+1}$ sorts converges to $1/3$ as $m \to \infty$.
\end{conj}

What Conjecture~\ref{conj:random} would mean is that, in the limit, any of these random chip-firing protocols sorts $\Delta^{2m+1}$, except that with probability $2/3$ the protocol does not locally sort the three chips which occupy the vertex it fires on its last move.

\subsection{Other types}
In this subsection we describe a ``Type B'' analog of our main result, which follows in a straightforward way from our main result via symmetry. Consider the following moves applied to a labeled configuration $\mathcal{C}$ on $\mathbb{Z}$:
\begin{enumerate}[(I)]
\item if chips~$(a)$ and $(b)$ with $a < b$ are both at vertex $i \in \mathbb{Z}$, move $(a)$ leftward one vertex and $(b)$ rightward one vertex (this is the usual labeled chip-firing move);
\item if chip~$(a)$ is at vertex $i \in \mathbb{Z}$ and chip~$(b)$ is at vertex~$-i$, move both~$(a)$ and~$(b)$ rightward one vertex;
\item if chip~$(a)$ is at the origin, move~$(a)$ rightward one vertex.
\end{enumerate}

\begin{thm} \label{thm:typeb}
For any $n\geq 1$, starting from the configuration~$\Delta^n$ and applying the moves~$(I)$, $(II)$, and $(III)$ in any order for as long as we can, we always arrive at the configuration that has chip~$(i)$ at vertex~$i$ for all $1 \leq i \leq n$.
\end{thm}
\begin{proof}
The proof follows from Theorem~\ref{thm:main} by considering a symmetric version of the labeled chip-firing process. Suppose we carry out the labeled chip-firing process on~$\mathbb{Z}$, starting from $2n$ chips at the origin with labels~$-n,-(n-1),\ldots,-1,1,2,\ldots,n$. Also, suppose that whenever we fire $(a)$ and $(b)$ together, where $a \neq -b$, we also immediately fire $(-a)$ and $(-b)$ together. This will mean that at all times in the process, if chip~$(a)$ is at vertex $i$ then chip~$(-a)$ must be at vertex $-i$ (i.e., the configuration will always be symmetric about the origin). Moreover, we claim that the way the positive labeled chips evolve in this process is exactly according to the moves (I), (II) and (III) above. Specifically, move (I) corresponds to firing $(a)$ and $(b)$  together (and $(-a)$ and $(-b)$ together). Move (II) corresponds to firing $(a)$ and $(-b)$ together, and $(-a)$ and $(b)$ together. Finally, move (III) corresponds to firing $(a)$ and $(-a)$ together, which because of the symmetry we maintain during the process, can only happen if $(a)$ and $(-a)$ are both at the origin. The claimed result then follows from Theorem~\ref{thm:main}, which says precisely that chip~$(i)$ ends at vertex~$i$ for all $1 \leq i \leq n$.
\end{proof}

Why do we call the process in Theorem~\ref{thm:typeb} a Type B analog of labeled chip-firing? First let us explain why the ordinary labeled chip-firing process is ``Type A.'' Suppose that~$\mathcal{C}$ is a labeled configuration on $\mathbb{Z}$ with label set $[n]$, and consider the vector~$v(\mathcal{C})$ in~$\mathbb{R}^n$ which records the positions of chips: $v(\mathcal{C}) \coloneqq (\mathcal{C}(1),\mathcal{C}(2),\ldots,\mathcal{C}(n))$. How does~$v(\mathcal{C})$ evolve as we cary out the usual labeled chip-firing process? Firing chips~$(i)$ and~$(j)$ with $i < j$ corresponds to adding the vector $e_j - e_i$ (where $e_i$ is the standard basis vector), and we are allowed to perform such a move whenever $\mathcal{C}(i) = \mathcal{C}(j)$, i.e., whenever $(v(\mathcal{C}),e_j-e_i) = 0$ (where $(\cdot,\cdot)$ denotes the standard inner product on~$\mathbb{R}^n$). The collection of vectors $\{e_j-e_i\colon 1\leq i < j \leq n\}$ is (one choice for) the set of positive roots of the root system of Type $A_{n-1}$. (Consult Bourbaki~\cite{bourbaki2002lie} for the basic theory of root systems.) Theorem~\ref{thm:main} says that, if $n$ is even, starting from the origin~$v \coloneqq (0,0,...,0) \in \mathbb{R}^n$ and repeatedly updating our vector $v$ by $v \mapsto v+\alpha$ whenever we have~$(v,\alpha) = 0$ for some positive root $\alpha$ of Type $A_{n-1}$, we always terminate at the same final vector no matter what choices we make along the way. Theorem~\ref{thm:typeb} is the exact same statement, except that we take our $\alpha$'s to be the positive roots of Type $B_n$ and the conclusion now holds for all values of $n$.

We thank Pavel Galashin for this ``vector-firing'' interpretation of our result. In ongoing research with Pavel Galashin and Alex Postnikov we are exploring the deeper connection between labeled chip-firing and the theory of root systems, and are attempting to extend this confluence result to other types. For some related work on chip-firing and root systems, see~\cite{benkart2016chip}.

\bibliography{chip_firing_sorting}{}

\begin{thebibliography}{10}

\bibitem{anderson1989disks}
Richard Anderson, L{\'a}szl{\'o} Lov{\'a}sz, Peter Shor, Joel Spencer, {\'E}va
  Tardos, and Shmuel Winograd.
\newblock Disks, balls, and walls: analysis of a combinatorial game.
\newblock {\em Amer. Math. Monthly}, 96(6):481--493, 1989.

\bibitem{bak1987self}
Per Bak, Chao Tang, and Kurt Wiesenfeld.
\newblock Self-organized criticality: An explanation of the 1/ \textit{f}
  noise.
\newblock {\em Phys. Rev. Lett.}, 59:381--384, 1987.

\bibitem{benkart2016chip}
Georgia Benkart, Caroline Klivans, and Victor Reiner.
\newblock Chip firing on {D}ynkin diagrams and {M}c{K}ay quivers.
\newblock Eprint published online at \arxiv{1601.06849}, 2016.

\bibitem{bjorner1991chip}
Anders Bj{\"o}rner, L{\'a}szl{\'o} Lov{\'a}sz, and Peter~W. Shor.
\newblock Chip-firing games on graphs.
\newblock {\em European J. Combin.}, 12(4):283--291, 1991.

\bibitem{bond2016abelian1}
Benjamin Bond and Lionel Levine.
\newblock Abelian networks {I}. {F}oundations and examples.
\newblock {\em SIAM J. Discrete Math.}, 30(2):856--874, 2016.

\bibitem{bond2016abelian2}
Benjamin Bond and Lionel Levine.
\newblock Abelian networks {II}: halting on all inputs.
\newblock {\em Selecta Math. (N.S.)}, 22(1):319--340, 2016.

\bibitem{bond2016abelian3}
Benjamin Bond and Lionel Levine.
\newblock Abelian networks {III}: {T}he critical group.
\newblock {\em J. Algebraic Combin.}, 43(3):635--663, 2016.

\bibitem{bourbaki2002lie}
Nicolas Bourbaki.
\newblock {\em Lie groups and {L}ie algebras. {C}hapters 4--6}.
\newblock Elements of Mathematics (Berlin). Springer-Verlag, Berlin, 2002.
\newblock Translated from the 1968 French original by Andrew Pressley.

\bibitem{corry2017divisors}
Scott Corry and David Perkinson.
\newblock Divisors and sandpiles.
\newblock In the \emph{AMS Student Mathematical Library} series, 2017.
\newblock Book in progress; draft available online at
  \url{http://people.reed.edu/~davidp/divisors_and_sandpiles/}.

\bibitem{dhar1990self}
Deepak Dhar.
\newblock Self-organized critical state of sandpile automaton models.
\newblock {\em Phys. Rev. Lett.}, 64(14):1613--1616, 1990.

\bibitem{dhar1999abelian}
Deepak Dhar.
\newblock The abelian sandpile and related models.
\newblock {\em Physica A: Statistical Mechanics and its Applications}, 263(1):4
  -- 25, 1999.

\bibitem{engel1976why}
Arthur Engel.
\newblock Why does the probabilistic abacus work?
\newblock {\em Educational Studies in Mathematics}, 7(1/2):59--69, 1976.

\bibitem{holroyd2015abelian}
Alexander~E. Holroyd, Lionel Levine, and Peter Winkler.
\newblock Abelian logic gates.
\newblock Eprint published online at \arxiv{1511.00422}, 2015.

\bibitem{huet1980confluent}
G{\'e}rard Huet.
\newblock Confluent reductions: Abstract properties and applications to term
  rewriting systems.
\newblock {\em J. ACM}, 27(4):797--821, October 1980.

\bibitem{levine2010sandpile}
Lionel Levine and James Propp.
\newblock What is {$\dots$} a sandpile?
\newblock {\em Notices Amer. Math. Soc.}, 57(8):976--979, 2010.

\bibitem{newman1942theories}
M.~H.~A. Newman.
\newblock On theories with a combinatorial definition of ``equivalence.''.
\newblock {\em Ann. of Math. (2)}, 43:223--243, 1942.

\bibitem{oeis2016}
N.J.A. Sloane.
\newblock The {O}n-{L}ine {E}ncyclopedia of {I}nteger {S}equences, 2016.
\newblock Published online at~\url{https://oeis.org/}.

\bibitem{spencer1986balancing}
J.~Spencer.
\newblock Balancing vectors in the max norm.
\newblock {\em Combinatorica}, 6(1):55--65, 1986.

\end{thebibliography}
\bibliographystyle{plain}

\end{document}